\documentclass[reqno, 11pt]{amsart}

\usepackage{anysize}
\marginsize{3cm}{3cm}{3cm}{3cm}
\usepackage[utf8]{inputenc}
\usepackage{hyperref}
\usepackage{enumitem}
\usepackage[margin=2.5cm]{geometry}
\usepackage{amsfonts,amsmath,amsthm,amssymb}
\usepackage{graphicx}
\usepackage{siunitx}
\usepackage{mathrsfs}
\usepackage{xcolor}

\numberwithin{equation}{section}
\newtheorem{theorem}{Theorem}[section]

\newtheorem{lemma}[theorem]{Lemma}
\newtheorem{proposition}[theorem]{Proposition}
\theoremstyle{definition}

\theoremstyle{remark}
\newtheorem{remark}[theorem]{Remark}

\DeclareMathOperator{\divop}{div}

\DeclareMathOperator{\N}{\mathbb{N}}
\DeclareMathOperator{\R}{\mathbb{R}}

\DeclareMathOperator{\T}{\mathbb{T}}
\DeclareMathOperator{\D}{\mathbb{D}}

\def\q{\quad}
\def\qq{\qquad}

\def\pa{\partial}
\newcommand{\pare}[1]{\left( #1 \right)}
\newcommand{\norm}[1]{\left\| #1 \right\|}
\newcommand{\av}[1]{\left| #1 \right|}
\newcommand{\bra}[1]{\left[ #1 \right]}

\renewcommand{\t}[1]{\text{#1}}

\def\eps{\varepsilon}
\allowdisplaybreaks
\newcommand{\ov}[1]{\overline{#1}}
\def\ww{\ov{w}}
\usepackage{enumitem}

\setlist[itemize]{leftmargin=*}
\newcommand{\dx}{\, {\mathrm d}}

\title[]{{Traveling Motility of Actin Lamellar Fragments Under  spontaneous symmetry breaking}}

\author[C. Garc\'ia]{Claudia Garc\'ia}
\address{ Departamento de Matem\'atica Aplicada \& Research Unit ``Modeling Nature'' (MNat), Facultad de Ciencias, Universidad de Granada, 18071 Granada, Spain}
\email{ claudiagarcia@ugr.es}

\author[M. Magliocca]{Martina Magliocca}
\address{Departamento de Análisis Matemático, Universidad de Sevilla, 41012 Sevilla, Spain}
\email{mmagliocca@us.es}

\author[N. Meunier]{Nicolas Meunier}
\address{LaMME, UMR 8071 CNRS, Université Evry Val d’Essonne, France.}
\email{nicolas.meunier@univ-evry.fr}

\begin{document}

\date{\today}

\begin{abstract}
{Cell motility is connected to the spontaneous symmetry breaking of a circular shape.  In \cite{BMC}, Blanch-Mercader and Casademunt perfomed a nonlinear analysis of the minimal model proposed by Callan and Jones \cite{12} and numerically conjectured the existence of traveling solutions once that symmetry is broken. In this work, we prove analytically that conjecture by means of nonlinear bifurcation techniques.}
\end{abstract}

\maketitle
\tableofcontents

\section{Introduction}
Cell migration is a fundamental process that is involved in many physiological and pathological functions (immune response, morphogenesis, cancer metastasis, etc.), and that is based on a complex intracellular machinery. Therefore, understanding its key features in spite of the variety of cellular behaviours is a challenging task. Recently, a biophysical approach showed that even if different migration modes coexist, cell migration obeys to a very general principle.

Actin filaments are components of the cell cytoskeleton, which is the dynamic set of biopolymers responsible for the integrity and force generation of the cell. In particular, these filaments are polar: they grow by polymerizing at one end, and shrink by depolymerizing at the other end. In migrating cells, polymerizing ends are located near the cell membrane, which resists the filaments’ growth. As a result, in the frame of reference of the cell, actin filaments drift away from the membrane, forming the so-called actin retrograde flows. In the cell motility framework, this mechanism can be explained in the following way: the cell adheres to the substrate, and the actin cytoskeleton, which moves within the cell, induces movements pushing forward the membrane by polymerizing actin \cite{17,AMM24,1,2,3,4,5}. These movements cause the cell to move forward as it detaches from the substrate at the back.

Existing physical models based on a fluid description of the cytoskeleton mainly consist in free boundary problems, see \cite{13,AMM5,AMM6}  example given, and aim at investigating the cell shape stability. In \cite{BMC}, the authors use a Darcy law for the fluid depicting the actin cytoskeleton and its polar order, in the context of a crawling cell. The authors model a cytoskeletal fragment and show the nonlinear instability of the center of mass of the system. They also relate the fragment’s shape to its velocity. The aim of this work is to study with greater mathematical rigor the model introduced in \cite{BMC}.

%In a general context,  polymerization  is defined as  any process in which relatively small molecules (monomers) combine chemically to produce a large chains or networks molecule  (polymers). In the cell motility framework, this mechanism can be explained in the following way: the cell adheres to the substrate, and the actin cytoskeleton, which moves within the cell, induces movements pushing forward the membrane by polymerizing actin \cite{17,AMM24,1,2,3,4,5}. These movements cause the cell to move forward as it detaches from the substrate at the back. An interesting picture showing this process is contained in \cite[Figure 1]{AE}.\\%  When no external cues are considered, this phenomena is called spontaneous polarization.\\
%Since we are interested in cell motility due to polymerization, we just quote \cite{13,AMM5,AMM6}  where cell motility is studied as consequence of cells contraction.
%\\

%\blu{(more biological context!!!!)}

%As already said, the model in \eqref{eq:model} describe crawling motility of cells and, in this work, we focus on motion by polymerization. \\

We now present the equations arising in the aforementioned model \cite{BMC,12}, together with their physical explanation. Let $D$ a bounded simply-connected domain in $\R^2$, $P$ the pressure, $v$ the normal velocity of the fluid, and $n$ the normal unit vector to the boundary $\partial D$.

The interaction among friction and viscous forces between cell and substrate is described by the Darcy's law
\begin{align*}
\xi v=-\nabla P, &\quad \mbox{ in }D,
\end{align*}
where  $\xi$ is the friction coefficient.
%The Darcy's law   \eqref{eq:model-P} describes either \cite{12}, or the fact that the actin layer is confined \cite{17}. \\
We assume that actin polymerizes on the boundary with normal direction, and that depolymerization occurs at a rate which is proportional to the filament density. Assuming that the filament density is constant (see \cite{12}), we get that
\begin{align*}
\nabla\cdot v=-k_d, &\quad \mbox{ in }D,\\
V_n=v\cdot n-v_p, &\quad \mbox{ on }\partial D,
\end{align*}
being $k_d$  the rate of the actin depolymerization, $v_p$ the polymerization speed, and $V_n$ the normal velocity to the interface. Note that $v_p$ acts like a source at the boundary.

We neglect the viscosity of outer fluids, so we can assume a Young-Laplace pressure drop across the boundary, that is,
\begin{align*}
P=\gamma\kappa, &\quad \mbox{ on }\partial D,
\end{align*}
where  $\gamma$ is the surface tension, and $\kappa$ the curvature.

Indeed, this model describes a free boundary problem since one aims to find the evolution of the boundary of $D$. It is for this reason that we change the notation writing $D(t)$ instead of just $D$ to emphasize its evolution in time. Hence, we find the following system:
\begin{subequations}\label{eq:model}
\begin{align}
\xi v=-\nabla P, &\quad \mbox{ in }D(t),\label{eq:model-P}\\
\nabla\cdot v=-k_d, &\quad \mbox{ in }D(t),\label{eq:model-v}\\
P=\gamma\kappa, &\quad \mbox{ on }\partial D(t),\label{eq:model-bc-P}\\
V_n=v\cdot n-v_p, &\quad \mbox{ on }\partial D(t).\label{eq:model-bc-v}
\end{align}
\end{subequations}

Both the biological explanation of the cell motility process and the physical justification of the model can be found in the Doctoral Thesis \cite{BM} (see, respectively, \cite[Section 1.2]{BM} and \cite[Section 2.2]{BM}).

%We now get into the mathematical aspects we propose to study, explaining which are our goals and justifying their interest.\\
Several works (see, for instance, \cite{AMM1,AMM24,AMM31,AMM32}, and also \cite{1,4,6,7,10})  analyzed the spontaneous symmetry breaking as consequence actin-based motility. Mathematically speaking, this translates into the existence of non-trivial traveling waves solutions. Thus,  our main purpose is proving the existence of traveling waves solutions to \eqref{eq:model}.

%We quote a few examples of the known literature concerning the proof of existence of traveling waves to cell motility models.\\
The authors of \cite{AMM} dealt with the case $v_p=k_d=0$ in \eqref{eq:model-v}-\eqref{eq:model-bc-v} and studied the case in which motion is induced by polymerization. The main difference between their model and \eqref{eq:model} is the fact that the boundary condition of the pressure \eqref{eq:model-bc-P} is coupled with a function depending on the polarity marker concentration $c=c(t,x,y)$,   whose time evolution verifies an advection-diffusion equation.
Cell motility models with cells moving by contraction can be found, for instance, in \cite{AMM4,AMM5,AMM6}.  In this cases, the boundary condition \eqref{eq:model-bc-P} has the same form and the coupling among pressure and myosin concentration appears in the divergence of equation \eqref{eq:model-P}.
%in   assume that cells move by contraction mechanisms due to myosin motors.

In the following, we state a formal version of our main theorem. A more detailed statement can be found in Theorem \ref{theo:main}.

\begin{theorem}\label{theo:main-intro}
For any $m\geq 2$, there exists $\xi\in I\mapsto (\gamma_{\xi},D_{\xi})$, with $D_{\xi}$ a $m$-fold symmetric domain, defining a traveling wave solution to \eqref{eq:model-P}--\eqref{eq:model-bc-v} with some constant speed.
\end{theorem}

{In Section 2, we study the linear stability of the rest state and reformulate the problem so that it is posed on a fixed boundary. Later, in Section 3 we perform a bifurcation analysis to find the existence of nontrivial traveling waves solutions. Finally, Section 4 gives the final proof of the main result of this work.
}

%\subsection{Notation}

\section{Preliminary results}

In this section, we study the linear stability of the rest state, that is the stationary state associated with zero velocity, whose shape is a disk. Then, we reformulate the problem of finding traveling wave solutions to \eqref{eq:model-P}--\eqref{eq:model-bc-v} in terms of a boundary equation via the use of the Hilbert transform. We shall also define our function spaces and provide the plan of the proof by means of the Crandrall-Rabinowitz theorem.

\subsection{Rest state}
We start by giving the expression of a rest state for \eqref{eq:model} in the case of the disk. By rest state, we mean stationary state with no displacement.

\begin{lemma}\label{lem:rset}
	Assume that $k_p$, $v_p$ and $R_0$ satisfy
	\[
	k_d R_0= 2v_p\, .
	\]
	In the case where the domain is the disk of radius $R_0$, the radial function
	\begin{equation*}%\label{eq:rest_state}
	P_0(r,\theta):= \frac{k_d}{4} (r^2-R_0^2) + \frac{\gamma}{R_0^2} , \end{equation*}
	is the unique stationary solution of \eqref{eq:model}.
\end{lemma}
\begin{proof}
	The fact that $P(r,\theta)= \frac{k_d}{4} (r^2-R_0^2) + \frac{\gamma}{R_0^2} $ is a stationary solution of \eqref{eq:model} is straightforward.

	In the case where $D$ is the disk of radius $R_0$ and the domain velocity is zero, \eqref{eq:model} rewrites for $\overline{P}(x,y):=P(x,y) - \frac{k_d}{4} (x^2 + y^2)$ as
	\[
	\begin{cases}
		\Delta \overline{P}(x,y) = 0\,& \quad \text{ in } \Omega \, ,\\
		\overline{P} = \frac{\gamma }{R_0} -\frac{k_d}{4}& \quad  \text{ on } \partial \Omega \,, \\
		0=-\left(\nabla \overline{P} + \frac{k_d}{2} \begin{pmatrix}
			x\\ y
		\end{pmatrix} \right)\cdot n-v_p, &\quad \mbox{ on }\partial \Omega,
	\end{cases}
	\]
	and the last condition simply reads as
	\[
	\nabla \overline{P}\cdot n = - \frac{k_dR_0}{2}
		 -v_p=0,
	\]
	according to the hypothesis made.

	By multiplying by $\overline{P}$ and integrating by parts, we obtain
	\[
	0=\int_{\Omega} \overline{P} \Delta \overline{P} \dx x \dx y = - \int_{\Omega} |\nabla \overline{P} |^2 \dx x \dx y,
	\]
	hence $\nabla \overline{P}=0$ on $D$ and the conclusion follows.
\end{proof}

\subsection{Linear stability of the rest state $P_0$}

\medskip
In order to study the previously found stationary state, we wish to linearize problem \eqref{eq:model} around this stationary state.
To do this, we perturb the stationary state. We can therefore assume that there exists a function
$R:\R_+\times (-\pi,\pi]\to \R_+$
such that for all $t>0$, we have
$D(t)=\{(r,\theta)\in \R_+\times (-\pi,\pi] \textrm{ such that } r<R(t,\theta)\}$.

We take a perturbation of the free boundary of the form
\[
r=  R_0 + \varepsilon \rho(t,\theta),
\]
i.e.
\begin{equation*}%\label{polaire_0}
	D(t)   = \left\{ (x,y)=\left( r \cos \theta ,  r \sin \theta \right); 0 \le r < R_0 + \varepsilon \rho(t,\theta)\right\} .
\end{equation*}

Let $\varepsilon  > 0$. The perturbation of the steady state is written as
\begin{eqnarray*}
	R(t,\theta )&=& R_0 + \varepsilon \rho(t,\theta), \\
	P(t,r,\theta)&=& P_0(r,\theta) + \varepsilon \widetilde{P}(t,r,\theta),\\
	u(t,r,\theta)&=&u_0(r,\theta)+\varepsilon \widetilde{u}(t,r,\theta),
\end{eqnarray*}
where
\[
P_0(r,\theta)= \frac{k_d}{4} (r^2-R_0^2) + \frac{\gamma}{R_0^2}.
\]
%	\begin{align*}
%		P(t,r,\theta)&=\widetilde{P} +\varepsilon Q(t,r,\theta) + \mathcal O (\varepsilon ^2).
%	\end{align*}
%	We derive the linearized problem associated to \eqref{eq:model} verified by $Q$ in the following Lemma.

\begin{proposition}\label{prop-ex-pb-lin}
	We can decompose $\rho$ into a Fourier series as follows
	\[
	\rho(t,\theta)=\sum_{m \ge 0}\left(\widetilde{R}_{\text{c}m} (t) \cos (m\theta)+\widetilde{R}_{\text{s}m}(t) \sin(m\theta)\right),
	\]
	where, for all $m\in \N$, the functions $\widetilde{R}_{\text{c}m}$ and $\widetilde{R}_{\text{s}m} $ satisfy the following ordinary differential equation
	\begin{equation}\label{eq:dispersion}
	Y'(t)= \left( \frac{k_d}{2}(m-1)-\gamma R_0^{-3}m (m^2-1)\right)  Y(t).
	\end{equation}
%	\begin{equation}\label{eq:lin_general}
%		\frac{\dx }{\dx t}
%			\rho =\mathcal I \left[ \frac{\gamma}{R_0} \left(\partial_{\theta \theta}^2\rho + \rho\right) +k_d R_0^2 \right],
%	\end{equation}
%	where $\mathcal A$ is the operator defined on $H^3\left(\partial B_{R_0}\right) \times H^2\left(B_{R_0}\right)$ by
%	\begin{equation*}%\label{eq:def_A}
%		\mathcal A :
%			\rho
%		\mapsto
%			\mathcal I \left[\frac{\gamma}{R_0} \left(\partial_{\theta \theta}^2\rho + \rho\right)+k_d R_0^2\right].
%	\end{equation*}

\end{proposition}

\begin{proof}

	Let us first expand the small perturbations in normal Fourier modes
	\[
		\rho(t,\theta)=\sum_{m \ge 0}\left(\widetilde{R}_{\text{c}m} (t) \cos (m\theta)+\widetilde{R}_{\text{s}m}(t) \sin(m\theta)\right) .
	\]
	By definition of the stationary state, we have $\divop u_0 = -k_d$.
	We also have $\divop u_0 = -k_d$, hence $\divop \widetilde{u} = 0$.
	Similarly, we have $u_0=-\nabla P_0$ and $u=-\nabla P$, hence $\widetilde{u}=-\nabla \widetilde{P}$. Therefore, we have $\Delta \widetilde{P}=0$. Hence it exists $a_n(t)$ and $b_n(t)$ such that
	\begin{equation*}
		\widetilde{P}(t,r,\theta) = \sum_m \left(a_m (t) r^m\cos(m\theta) + b_m (t) r^m \sin(m\theta)\right). \label{eq:delta p 1}
	\end{equation*}

%	From the normal force balance, it follows that $\widetilde{P} =\gamma \widetilde{\kappa}$ on $\partial\Omega$. On the boundary, the linear deviations in the fluid pressure and solute concentration are
%	\begin{eqnarray*}
%		\delta p & \simeq \delta p(r=1) = \sum_{m} \left( A_{m} (t) \cos (m\theta) + B_{m}  (t) \sin (m\theta) \right) \label{eq:delta p boundary}
%	\end{eqnarray*}
%

	Firstly, we compute the linearization of the curvature.  The curve
	\[
	\gamma_t(\theta) = (R(t, \theta)\cos(\theta),R(t, \theta) \sin(\theta)),
	\]
	is a parameterization of the boundary of $D(t)$. Hence, the curvature is given by
	\begin{equation*}%\label{eq:courbure}
		\kappa(g)= \frac{\det(\gamma_t'(\theta),\gamma_t''(\theta))}{\|\gamma_t'(\theta)\|^3}=\frac{R(t, \theta)^2+2(\partial_\theta R(t, \theta)) ^2- R(t, \theta) \partial_{\theta\theta} R(t, \theta)}{\left(R(t, \theta)^2+(\partial_\theta R(t, \theta)) ^2\right)^{3/2}}.
%		\simeq 1-\delta R(t,\theta)-\partial_{\theta\theta}\delta R(t,\theta)
	\end{equation*}
	Set $\kappa(t,\theta)= \kappa_0+\varepsilon \widetilde{\kappa}(t,\theta)$. Since $P_{|\partial D(t)}=\gamma \kappa$, and
	\begin{eqnarray*}
		P_{0|\partial D(t)}&=& P_0(R(t,\theta),\theta)\\
		&=& \frac{\gamma}{R_0} -\frac{k_d}{4}R_0^2 +\frac{k_d}{4}\left( R(t,\theta)\right) ^2 \\
		&=&  \frac{\gamma}{R_0} +\varepsilon \frac{k_d}{2}R_0 \rho(t,\theta)
		 + o(\varepsilon^2),
	\end{eqnarray*}
	we have
	\begin{eqnarray*}
		P_{|\partial D(t)}&=& P_{0|\partial D(t)} +\varepsilon \widetilde{P}_{|\partial D(t)}\\
		&=& \gamma \kappa_0  +\varepsilon \frac{k_d}{2} R_0 \widetilde{R}(t,\theta)  +\varepsilon \widetilde{P}_{|\partial D(t)}
		+ o(\varepsilon^2)\\
		&=& \gamma \kappa_0  +\gamma \varepsilon \widetilde{\kappa},
	\end{eqnarray*}
	from which we deduce
	\[
	\widetilde{P}_{|\partial D(t)} = \gamma \widetilde{\kappa}-  \frac{k_d}{2} R_0 \widetilde{R}(t,\theta)  .
	\]

	Moreover, we compute to the first order in $\varepsilon$,
	\begin{eqnarray*}
	\kappa\left(R_0+\varepsilon \rho (t,\theta)\right) &=&\frac1R_0-\frac{\varepsilon}{R_0^2} \left(\partial ^2_{\theta \theta}\rho(t,\theta) +\rho(t,\theta)\right)+o (\varepsilon^2)\\
	&=&\kappa_0-\frac{\varepsilon}{R_0^2} \left(\partial ^2_{\theta \theta}\rho(t,\theta) +\rho(t,\theta)\right)+o (\varepsilon^2),
	\end{eqnarray*}
	hence
	\begin{eqnarray*}
		\widetilde{\kappa}(t,\theta)&=&-\frac{1}{R_0^2} \left(\partial ^2_{\theta \theta}\rho(t,\theta) +\rho(t,\theta)\right)+o (\varepsilon) \nonumber\\
		&=&\frac{1}{R_0^2} \sum_{m\ge 0} (m^2-1)\left(\widetilde{R}_{\text{c}m} (t) \cos (m\theta)+ \widetilde{R}_{\text{s}m}(t) \sin(m\theta)\right)+o (\varepsilon). \label{eq:delta kappa}
	\end{eqnarray*}

	Furthermore,
	\begin{eqnarray*}
		\widetilde{P}(t,r,\theta)_{|\partial D(t)}&=& \widetilde{P}(t,R_0,\theta)+o(\varepsilon)\\
		&=& \sum_{m\ge 0} R_0^m\left(a_m (t) \cos (m\theta)+ b_m(t) \sin(m\theta)\right)+o (\varepsilon),
	\end{eqnarray*}
	and we deduce that for all $m\in \N$
	\begin{eqnarray}
		a_m(t)=\left(\gamma (m^2-1)R_0^{-2-m}-\frac{k_d}{2}R_0^{1-m}\right) \widetilde{R}_{\text{c}m} (t), \label{eq:Am} \\
		b_m(t)=\left(\gamma (m^2-1)R_0^{-2-m}-\frac{k_d}{2}R_0^{1-m}\right) \widetilde{R}_{\text{s}m} (t). \label{eq:Bm}
	\end{eqnarray}

	We proceed by expressing a linearized version of the kinematic condition $V_\text{n}=u\cdot n-v_p$ on $\partial D(t)$. This task consists in finding the flow $u$, the normal $n$, and the velocity of the sharp interface $V_\text{n}$ in terms of of the linear perturbations.
	The flow is given by $u=-\nabla \delta P $, and thus
	\[
	u\cdot n= u_0\cdot n + \varepsilon \widetilde{u}\cdot n.
	\]
	Since $\{(r,\theta)\in \R_+ \times (-\pi,\pi] \textrm{ s.t. } r-R(t,\theta)=0)\}$ defines $\partial D(t)$, we have
	\[
	n(t,\theta)=\frac{\nabla\left(r-R(t,\theta)\right)}{\|\nabla\left(r-R(t,\theta)\right)\|}.
	\]
	Moreover,
	\begin{eqnarray*}
		\|\nabla\left(r-R(t,\theta)\right)\|^{-1}&=&\left( 1+\left( \frac{\varepsilon}{r}\partial_\theta \rho(t,\theta)\right)^2\right)^{-1/2}\\
		&=&1+o(\varepsilon^2),
	\end{eqnarray*}
	we have
	\[
	n(t,\theta)=\left(1+o(\varepsilon^2)\right)e_r +o(\varepsilon) e_\theta
	=n_r e_r +n_\theta e_\theta.
	\]
	Furthermore, we know that
	\[
	\widetilde{u}(t,r,\theta)=-\nabla \widetilde{P}(t,r,\theta)= -\partial_r \widetilde{P}(t,r,\theta)e_r-\frac{1}{r} \partial_\theta \widetilde{P}(t,r,\theta)e_\theta,
	\]
	thus
	\begin{equation}
			\left(\widetilde{u}\cdot n\right)_{|\partial D(t)} =   -\sum_m m R_0^{m-1}\left(a_m(t) \cos(m\theta)+b_m\sin(m \theta)\right) +o(\varepsilon).
		\label{eq:lin u}
	\end{equation}

	On the other hand, by definition of the stationary state, we have
	\begin{eqnarray*}
		u_0(r,\theta) &=& -\nabla P_0(r,\theta)\\
		&=& -\partial_r P_0(r,\theta) e_r-\frac1r\partial_\theta P_0(r,\theta) e_\theta\\
		&=&-\frac{k_d}{2} r e_r,
	\end{eqnarray*}
	hence
	\begin{equation}\label{eq:lin u_0}
	\left(u_0\cdot n\right)_{|\partial D(t)}=-\frac{k_d}{2}R_0-\frac{k_d}{2} \varepsilon \rho(t,\theta)+o(\varepsilon^2).
	\end{equation}

	Using Eqs. \eqref{eq:lin u} -- \eqref{eq:lin u_0}, we compute the normal fluid velocity to linear order
	\begin{eqnarray*}
		\left(u\cdot n\right)_{|\partial D(t)}&=&-\frac{k_d}{2}R_0-\varepsilon \frac{k_d}{2}\sum_{m\ge 0} \left(\widetilde{R}_{\text{c}m} (t) \cos (m\theta)+ \widetilde{R}_{\text{s}m} (t)\sin(m\theta)\right)\\
		& &-\varepsilon \sum_{m\ge 0}  m R_0^{m-1} \left(a_m(t) (t) \cos (m\theta)+ b_m(t) \sin(m\theta)\right).
	\end{eqnarray*}
	Using next \eqref{eq:Am}  and \eqref{eq:Bm}, we have
	\begin{eqnarray*}
		\left(u\cdot n\right)_{|\partial D(t)}&=&-\frac{k_d}{2}R_0-\varepsilon \frac{k_d}{2}\sum_{m\ge 0} \left(\widetilde{R}_{\text{c}m} (t) \cos (m\theta)+ \widetilde{R}_{\text{s}m} (t)\sin(m\theta)\right)\\
		& &-\varepsilon \gamma R_0^{-3}\sum_{m\ge 0}  m  (m^2-1)\left( \widetilde{R}_{\text{c}m} (t)\cos (m\theta)+ \widetilde{R}_{\text{s}m} (t)\sin(m\theta)\right)\\
		& & \qquad +\varepsilon \frac{k_d}{2} \sum_{m\ge 0}  m  \left( \widetilde{R}_{\text{c}m} (t)\cos (m\theta)+ \widetilde{R}_{\text{s}m} (t)\sin(m\theta)\right)+o(\varepsilon^2).
	\end{eqnarray*}

	Consequently, on the first hand we have
	\begin{eqnarray*}
		V_\text{n} &=& v_p-\frac{k_d}{2}R_0-\varepsilon \frac{k_d}{2} \sum_{m\ge 0} \left(\widetilde{R}_{\text{c}m} (t) \cos (m\theta)+ \widetilde{R}_{\text{s}m} (t)\sin(m\theta)\right)\\
		& &-\varepsilon \gamma R_0^{-3}\sum_{m\ge 0}  m  (m^2-1)\left( \widetilde{R}_{\text{c}m} (t)\cos (m\theta)+ \widetilde{R}_{\text{s}m} (t)\sin(m\theta)\right)\\
		& & \qquad +\varepsilon \frac{k_d}{2} \sum_{m\ge 0}  m  \left( \widetilde{R}_{\text{c}m} (t)\cos (m\theta)+ \widetilde{R}_{\text{s}m} (t)\sin(m\theta)\right)+o(\varepsilon^2),
	\end{eqnarray*}
	and on the other hand, since $V_\text{n}$ is the normal boundary velocity, we have
	\[
	V_\text{n} n_r = \frac{d R(t,\theta)}{dt} =\partial_t R(t,\theta)+ \frac{V_\text{n} n_\theta}{ R(t,\theta)} \partial_\theta R(t,\theta),
	\]
	and thus
	\begin{eqnarray*}
		V_\text{n}&=&\frac{\partial_t \rho(t,\theta)}{n_r-\frac{\partial_\theta \rho(t,\theta) }{1+\rho(t,\theta)}n_\theta}\\
		&=& \partial_t \rho(t,\theta)+o(\varepsilon^2)\\
		&=&\sum_{m}\left(\partial_t \widetilde{R}_{\text{c}m} (t) \cos (m\theta)+ \partial_t \widetilde{R}_{\text{s}m} (t) \sin(m\theta)\right) +o(\varepsilon^2).
	\end{eqnarray*}

	The term $v_p-\frac{k_d}{2}R_0$ is zero according to the assumption made. We deduce that for all $m\in \N$
	\begin{align*}
		\partial_t \widetilde{R}_{\text{c}m} (t)=\left( \frac{k_d}{2}(m-1)+\gamma R_0^{-3}m (m^2-1)\right)  \widetilde{R}_{\text{c}m} (t), %\label{eq:dRcos/dt}
\\
		\partial_t \widetilde{R}_{\text{s}m} (t)=\left( \frac{k_d}{2}(m-1)-\gamma R_0^{-3}m (m^2-1)\right)  \widetilde{R}_{\text{s}m} (t). %\label{eq:dRsin/dt}
	\end{align*}

\end{proof}

\begin{remark}
	Equation \eqref{eq:dispersion} shows that the stabilizing effect of surface tension is proportional to $m^3$ for large $m$.
\end{remark}

\begin{proposition}
%	A bifurcation occurs for those values of 1⁄4n 1=1⁄22nðnþ1Þ,with n>1, for which !ðnÞ 1⁄4 0.
 In the case where there exists an integer $n$ such that $\frac{2\gamma}{k_d R_0^3}=\frac{1}{n(n+1)}$, all modes $m<n$ are unstable in \eqref{eq:dispersion}.
\end{proposition}

\begin{proof}
	We observe that
	\begin{eqnarray*}
	\frac{k_d}{2}(m-1)-\gamma R_0^{-3}m (m^2-1) =(m-1)\left(\frac{k_d}{2}-\gamma R_0^{-3}m (m+1)\right)  ,
	\end{eqnarray*}
	from which the assertion follows.
\end{proof}

\subsection{Reformulation of the problem and reduction to the fixed boundary problem}

Let us reformulate the problem \eqref{eq:model} in a suitable way. After the transformation
$$
v=v'-\kappa_d\frac{r}{2},\quad P=P'+\xi k_d\frac{|r|^2}{4},
$$
we reduce the system \eqref{eq:model} to
\begin{align*}
\xi v=-\nabla P, &\quad \mbox{ in }D(t),\\
\nabla\cdot v=0, &\quad \mbox{ in }D(t),\\
P=\gamma\kappa-\frac{\xi k_d}{4}|r|^2, &\quad \mbox{ on }\partial D(t),\\
V_n=v\cdot n-\frac{k_d}{2}r\cdot n-v_p, &\quad \mbox{ on }\partial D(t).
\end{align*}
The third equation is known as the stress-free boundary condition, and the fourth one as the kinematic equation.

Next, let us normalize the constants in terms of the surface tension. For that, consider $k_d=2v_p=\xi=1$ and $\beta=\frac{k_d\gamma}{\xi v_p^2}=4\gamma$, where $\beta$ will be a free parameter in $\R$. Hence, the above system agrees with
\begin{align*}
v=-\nabla P, &\quad \mbox{ in }D(t),\\
\nabla\cdot v=0, &\quad \mbox{ in }D(t),\\
P=\frac{\beta}{4}\kappa-\frac{1}{4}|r|^2, &\quad \mbox{ on }\partial D(t),\\
V_n=v\cdot n-\frac{1}{2}r\cdot n-\frac12, &\quad \mbox{ on }\partial D(t).
\end{align*}

From the incompressibility condition, we know that $v=\nabla^\perp\psi$, where $\psi$ is called the stream function. Then, using the first equation we arrive to
$$
v=\nabla^\perp \psi=-\nabla P, \quad \mbox{in } D(t).
$$
Then $\psi$ and $P$ are harmonic conjugates, and its value at the boundary is related through the Hilbert transform $\mathcal{H}$ as
$$
\psi=\mathcal{H}[P], \quad \mbox{on } \partial D(t),
$$
where
\begin{equation}\label{eq:hilbert}
\mathcal{H}[f(e^{is})](e^{i\theta})=\frac{1}{2\pi}\int_0^{2\pi}f(e^{is})\cot((\theta-s)/2)ds.
\end{equation}

%
%\blu{(remove! add later)In particular, we will exploit its continuity,   i.e. if $f\in C^{n,\alpha}(\mathbb{T})$, then  $\mathcal{H}[f]\in C^{n,\alpha}(\mathbb{T})$, see \cite[Theorem 1]{BP} (and \cite[Section 1]{C}).}
%

Notice that hence the equations in $D$ are solved via the Hilbert transform, and it remains to study the equations at the boundary. Then, we arrive now to a free boundary problem for $\partial D$:
\begin{align*}
P=\frac{\beta}{4}\kappa-\frac{1}{4}|r|^2, &\quad \mbox{ on }\partial D(t),\\
V_n=v\cdot n-\frac{1}{2}r\cdot n-\frac12, &\quad \mbox{ on }\partial D(t),\\
\psi=\mathcal{H}[P], &\quad \mbox{on } \partial D(t).
\end{align*}

Note that the first term of the normal interface velocity can be computed through the Hilbert transform \eqref{eq:hilbert}
$$
n\cdot v=\vec{t}\cdot \nabla\psi=\frac{1}{|z'(\theta)|}\partial_\theta \psi(z(\theta))=\frac{1}{|z'(\theta)|}\partial_\theta \mathcal{H}[P](z(\theta))=-\frac14 \frac{1}{|z'(\theta)|} \partial_\theta \mathcal{H}[|r|^2-\beta\kappa ](z(\theta)),
$$
where $z(\theta)$ is a parametrization of $\partial D(t)$.

In order to   {reduce the free boundary problem to a fixed boundary one,} let us parametrize it using a conformal map \cite{13}. From the Riemann mapping theorem, we know that if $D(t)\neq \R^2$ is a nonempty bounded simply connected domain, there is a unique conformal map $\Phi_t:\D\mapsto D(t).$ In particular, $\Phi_t$ maps $\T$ into $\partial D(t)$ and hence we have the following parametrization for the boundary:
$$
\theta\mapsto \Phi_t(e^{i\theta}).
$$
Using the parametrization we write
\begin{align*}
r=&\Phi_t(w)=\Phi_t(e^{i\theta}), \quad w\in \T, \theta\in[0,2\pi],\\
n=&-\frac{w\Phi'_t(w)}{|\Phi'_t(w)|},\\
-\frac12 r\cdot n=&-\frac12 \Phi_t(w)\cdot \left(-\frac{w\Phi'_t(w)}{|\Phi'_t(w)|}\right)=\frac{1}{2|\Phi'_t(w)|}\mbox{Re}\left[w\overline{\Phi_t(w)}\Phi'_t(w)\right].
\end{align*}
Hence the equations follows as
\begin{align*}
P=&\frac{\beta}{4} \kappa[\Phi_t]-\frac14 |\Phi_t(w)|^2,\\
V_n=&-\frac14 \frac{1}{|\Phi'_t(w)|} \partial_\theta \mathcal{H}\left\{|\Phi_t|^2-\beta \kappa[\Phi_t] \right\}(\Phi_t(e^{i\theta}))+\frac{1}{2|\Phi'_t(w)|}\mbox{Re}\left[w\overline{\Phi_t(w)}\Phi'_t(w)\right]-\frac12.
\end{align*}
On the other hand, the curvature can be written as
$$
\kappa[\Phi_t](w)=\frac{1}{|\Phi_t'(w)|}\mbox{Re}\left[1+w\frac{\Phi_t''(w)}{\Phi_t'(w)}\right].
$$

Now, $V_n$ (the normal velocity to the interface) can be written as
$$
V_n=-\partial_t \Phi_t \cdot \frac{w\Phi'_t(w)}{|\Phi'_t(w)|}=-\frac{1}{|\Phi'_t(w)|}\mbox{Re}\left[\overline{\partial_t \Phi_t}w\Phi'_t(w)\right].
$$
Notice in the above formula that $\Phi_t$ depends on $t$. However, here we will be interested in solutions that translate along the horizontal axis and then
$$
\Phi_t(w)=\phi(w)+Vt,\quad \partial_t \Phi_t=V\in\R,
$$
for some $V\in\R$, obtaining
$$
V_n=-\frac{1}{|\phi'(w)|}\mbox{Re}\left[Vw\phi'(w)\right].
$$
The idea of the work will be to perform a perturbation argument around the unit disc, which happens to be a trivial traveling wave solution (indeed, we will prove that it is stationary). Hence, we shall write the conformal map $\phi$ as
$$
\phi(w)=w+\mu w+f(w), \quad w\in\T,
$$
with $\mu\in\R$ and where
\[
f(w)=\sum_{n\geq 2} a_n w^{n+1},
\]
and $a_n\in\R$. Hence, using the translational symmetry of the kinematic condition,  it agrees with
$$
F(V,\beta,\mu, f)(\theta)=0, \quad \theta\in[0,2\pi],
$$
being $F$ defined as
\begin{equation}\label{eq:F}
\begin{aligned}
&F(\beta,V,\mu, f)(\theta)\\
&=\mbox{Re}\left[Vw\phi'(w)+\frac{1}{4}\beta\partial_\theta \mathcal{H}[\kappa[\phi(w)]](\theta)-\frac14\partial_\theta\mathcal{H}[|\phi(w)|^2](\theta)   +\frac12 w \overline{\phi(w)}\phi'(w)-\frac12 |\phi'(w)|\right],
\end{aligned}
\end{equation}
with
$$
\kappa[\phi(w)]=\frac{1}{|\phi'(w)|}\mbox{Re}\left[1+w\frac{\phi''(w)}{\phi'(w)}\right].
$$

\begin{remark}
  {Notice that from the definition of $f$ we are excluding the second Fourier mode. That is coherent with Casademunt work, where they find some degeneracy in the second mode. Here we can exclude it directly from the function spaces.} Moreover, we shall prove that the only disc being a trivial solution to $F=0$ is the unit one. However, we need to add a dilatation of the disc in the perturbed solution, this is represented by the constant $\mu\in\R$ above.
\end{remark}

In the following proposition, we check that the disc is a stationary solution.
\begin{proposition}\label{prop:0}
If $  {D(0)}$ is the unit disc, then it is a stationary solution. That means the following
$$F(\beta,0,0,0)=0,\quad  \beta\in\R.$$
\end{proposition}
\begin{proof}
Notice that
$$
F(\beta,0,0,0)(\theta)=\mbox{Re}\left[\frac{1}{4}\beta\partial_\theta \mathcal{H}[1](\theta)-\frac14\partial_\theta\mathcal{H}[1](\theta)   +\frac12 -\frac12 \right].
$$
Since the Hilbert transform of a constant vanishes, we easily get that $F(\beta,0,0,0)\equiv 0$.
\end{proof}

Moreover, following the computations above we find
\begin{equation*}\label{Fmu}
F(\beta,0,\mu,0)(\theta)=\frac12 (1+\mu)\mu,
\end{equation*}
which is only vanishing for $\mu=0$, meaning that $\Phi(w)=w$, or for $\mu=-1$ referring to $\Phi(w)=0$, which is not possible. Hence, the only possible trivial solution happens to $\mu=0$.

\subsection{Function spaces and well-definition of $F$}

Let us emphasize again that the functional $F$ is invariant under translations, which is a consequence of the translational symmetry of the kinematic condition stated before. That is the reason to exclude the constants from the definition of $f$.

\begin{proposition}\label{prop:trans}
The functional $F$ is invariant under translations, that is,
$$
F(\beta,V,\mu,f+a)=F(\beta,V,\mu,f), \quad a\in\R.
$$
\end{proposition}
\begin{proof}
Denote by $\phi_a=\phi_0+a$, being $\phi_0=(1+\mu)w+f$. Note that  $\phi_a'=\phi_0'$.
Hence
\begin{align*}
&F(\beta,V,\mu, f+a)(\theta)\\
&=\mbox{Re}\left[Vw\phi_0'(w)+\frac{1}{4}\beta\partial_\theta \mathcal{H}[\kappa[\phi_a]](\theta)-\frac14\partial_\theta\mathcal{H}[|\phi_a(w)|^2](\theta)  +\frac12 w \overline{\phi_a(w)}\phi_0'(w)-\frac12 |\phi_0'(w)|\right].
\end{align*}
Notice that
$$
|\phi_a(w)|^2=|\phi_0|^2+2a\textnormal{Re}[\phi_0(w)]+a^2,
$$
and since the Hilbert transform of constants vanishes, we do not see the contribution of the last term, that implies
\begin{align*}
&F(\beta,V,\mu, f+a)(\theta)\\
&=\mbox{Re}\left[Vw\phi_0'(w)+\frac{1}{4}\beta\partial_\theta \mathcal{H}[\kappa[\phi_a]](\theta)-\frac14\partial_\theta\mathcal{H}[|\phi_0(w)|^2](\theta)-\frac12 a \partial_\theta\mathcal{H}[\textnormal{Re}[\phi_0]](\theta) +a\frac12 w \phi_0'(w)\right.\\
&\q\left.  +\frac12 w \overline{\phi_0(w)}\phi_0'(w)-\frac12 |\phi_0'(w)|\right].
\end{align*}
For the same reason, we have that $\kappa[\phi_a]=\kappa[\phi_0]$, then
\begin{align*}
F(\beta,V,\mu, f+a)(\theta)=\mbox{Re}&\left[Vw\phi_0'(w)+\frac{1}{4}\beta\partial_\theta \mathcal{H}[\kappa[\phi_0]](\theta)-\frac14\partial_\theta\mathcal{H}[|\phi_0(w)|^2](\theta)-\frac12 a \partial_\theta\mathcal{H}[\textnormal{Re}[\phi_0]](\theta) \right.\\
&\q\left.  +a\frac12 w \phi_0'(w)+\frac12 w \overline{\phi_0(w)}\phi_0'(w)-\frac12 |\phi_0'(w)|\right].
\end{align*}
In the following, let us check that
\begin{equation}\label{translat-estim}
 \textnormal{Re}[w\phi_0'(w)]-\partial_\theta \mathcal{H}[\textnormal{Re}[\phi_0]](\theta)=0,
\end{equation}
in order to check that it does not depend on $a$. We will check it by using its Fourier expression. Notice that
$$
\phi_0(w)=(1+\mu)w+\sum_{n\geq 2}a_n w^{n+1},
$$
and thus
$$
w\phi_0'(w)=(1+\mu)w+\sum_{n\geq 2}a_n(n+1)w^{n+1}.
$$
That implies
\begin{equation*}\label{eq:Rewphi'}
\textnormal{Re}[w\phi_0'(w)]=(1+\mu)\cos(\theta)+\sum_{n\geq 2}a_n(n+1)\cos((n+
1)\theta).
\end{equation*}
On the other hand, using the result in  \cite[Section 9.3.1, eq. (9.3.8)]{BN}, we get
$$
\mathcal{H}[\textnormal{Re}[\phi_0]](\theta)=(1+\mu)\sin(\theta)+\sum_{n\geq 2}a_n \sin((n+1)(\theta),
$$
and thus
$$
\partial_\theta \mathcal{H}[\textnormal{Re}[\phi_0]](\theta)=(1+\mu)\cos(\theta)+\sum_{n\geq 2}a_n(n+1)\cos((n+1)\theta),
$$
implying \eqref{translat-estim}. Hence, we conclude $F(\beta,V,\mu, f+a)\equiv F(\beta,V,\mu, f)$.
\end{proof}

For $\alpha\in(0,1)$, define the following function spaces

\begin{align}
X^\alpha:=&\left\{f\in C^{3,\alpha}(\T), \quad f(w)=\sum_{n\geq 2}a_n w^{n+1}, \quad a_n\in\R\right\}\label{eq:X}\\
Y^\alpha:=&\left\{f\in C^{0,\alpha}(\T), \quad f(e^{i\theta})=\sum_{n\geq 0}a_n \cos(n\theta), \quad a_n\in\R\right\}\label{eq:Y}
\end{align}
Let us also define $B_{X^{  {\alpha}}}(\rho)$ as the ball in $X^{  {\alpha}}$ centered at 0 of radius $\rho.$

The main difficulty of this work relies on the function spaces. We shall observe later that we can write the linearized operator as
$$\partial_{f} F(\beta,0,0,0)[h](\theta)=\sum_{n\geq 2}\tilde{F}_n \cos(n\theta),$$
however we can not exclude the zero and first Fourier modes in the expression of the nonlinear operator $F$. That implies that the range of the linearized operator will not be closed in the range space, which does not agree with the needed properties to perform a perturbative argument.

In order to tackle that problem, we use the free constants $\mu$ referring to dilatations of the disc and $V$ related to the speed. Instead of working with $\partial_f F$, we shall include in the linearized operator derivatives with respect to $\mu$ and $V$. That will help us to find non trivial zero and first Fourier modes in the linearized operator.

In the following proposition we check that ${F}$ is well-defined and $C^1$.

\begin{proposition}\label{prop:welldefined}
For $\rho<1$, the operator
$ F:\R^3 \times   {B_{X^\alpha}(\rho)}\rightarrow Y^\alpha$  is well-defined and $C^1$.
\end{proposition}

\begin{proof}

We split this proof in three parts. We first prove that ${F}\in C^{0,\alpha}$ if $f\in X^\alpha$   {with $\norm{f}_{X^\alpha}\le \rho<1$}, and then that  ${F}\in Y^\alpha$. In the last part of this proof, we show the $C^1$ regularity.\\

\noindent
{\it $\bullet$ Step 1:  ${F}\in C^{0,\alpha}$.}\\
We recall the expression of $F$ in \eqref{eq:F} with
\begin{align*}
&{F}(\beta,V,\mu,f)(\theta)\\
&=\mbox{Re}\left[Vw\phi'(w)+\frac{1}{4}\beta\partial_\theta \mathcal{H}[\kappa[\phi(w)]](\theta)-\frac14\partial_\theta\mathcal{H}[|\phi(w)|^2](\theta)   +\frac12 w \overline{\phi(w)}\phi'(w)-\frac12 |\phi'(w)|\right],
\end{align*}
and we observe that
\begin{equation*}\label{eq:F-terms}
\mbox{Re}\bra{w\phi'(w)},\qquad\frac12 \mbox{Re}\bra{w \overline{\phi(w)}\phi'(w)},\qquad \frac12 |\phi'(w)|
\end{equation*}
 are $C^{2,\alpha}$ because $\phi=(1+{\mu})w+f$ by definition of $\phi$.

Now, we use the continuity of the Hilbert transform,   i.e. if $f\in C^{n,\alpha}(\mathbb{T})$, then  $\mathcal{H}[f]\in C^{n,\alpha}(\mathbb{T})$, see \cite[Theorem 1]{BP} (and \cite[Section 1]{C}). Hence, it implies that, since $\kappa[\phi]$ is in  $C^{1,\alpha}(\mathbb{T})$, the same holds for its Hilbert transform. Consequently, its derivative belongs to $C^{0,\alpha}(\mathbb{T})$.

\medskip\noindent
{\it $\bullet$ Step 2:  ${F}\in Y^{\alpha}$.}\\
We need to prove that $F$ can be decompose as a Fourier series in consines, that is,
$$
F(\beta,V, \mu,f)=\sum_{n\geq 0}F_n\cos(n\theta), \quad F_n\in\R.
$$
This can be done showing that
\[
F(\beta,V, \mu,f)(\theta)=F(\beta,V, \mu,f)(-\theta).
\]
The expression of $F(\beta,V, \mu,f)(-\theta)$ is given by
\begin{align*}
&F(\beta,V, \mu,f)(-\theta)\\
&=\mbox{Re}\left[V\ww\phi'(\ww)+\frac{1}{4}\beta\partial_\theta \mathcal{H}[\kappa[\phi(w)]](-\theta)-\frac14\partial_\theta\mathcal{H}[|\phi(w)|^2](-\theta)   +\frac12 \ww \overline{\phi(\ww)}\phi'(\ww)-\frac12 |\phi'(\ww)|\right],
\end{align*}
Since
\[
\phi(\overline{w})=\overline{\phi(w)},
\]
we  have that %the terms in \eqref{eq:F-terms} verifies
\begin{align*}
\mbox{Re}\bra{\overline{w}\phi'(\overline{w})}&=\mbox{Re}\bra{\overline{w\phi'(w)}}=\mbox{Re}\bra{w\phi'(w)},\nonumber
\\
\mbox{Re}\bra{\ww \overline{\phi(\overline{w})}\phi'(\overline{w})}&= \mbox{Re}\bra{\overline{w \overline{\phi(w)}\phi'(w)}}= \mbox{Re}\bra{w \overline{\phi(w)}\phi'(w)},%\label{eq:cos1}
\\
 |\phi'(\overline{w})|&=|\overline{\phi'(w)}|= |\phi'(w)|.%\label{eq:cos2}
\end{align*}
 We now focus on
\[
\partial_\theta\mathcal{H}[|\phi(w)|^2](-\theta)  ,\qq
\partial_\theta \mathcal{H}[\kappa[\phi(w)]](-\theta) ,
\]
beginning  with
\begin{align*}
\mathcal{H}[|\phi(w)|^2](-\theta) &=-\frac{1}{2\pi}\int_{0}^{2\pi}\av{\phi\pare{e^{is}}}^2\cot\pare{\theta+s}\,ds
\\
&=-\frac{1}{2\pi}\int_{0}^{2\pi}\av{\phi\pare{e^{-i\tilde{s}}}}^2\cot\pare{\theta-\tilde{s}}\,d\tilde{s}
\\
&=-\mathcal{H}[|\phi(\ww)|^2](\theta)
\\
&=-\mathcal{H}[|\phi(w)|^2](\theta).
\end{align*}
Reasoning in the same way, we have that
\[
\mathcal{H}[\kappa[\phi(w)]](-\theta) =-\mathcal{H}[\kappa[\phi(\ww)]](\theta) =-\mathcal{H}[\kappa[\phi(w)]](\theta)
\]
because
\[
\kappa[\phi(\ww)]=\frac{1}{|\phi'(\ww)|}\mbox{Re}\left[1+\ww\frac{\phi''(\ww)}{\phi'(\ww)}\right]=\frac{1}{|\phi'(w)|}\mbox{Re}\left[1+w\frac{\phi''(w)}{\phi'(w)}\right]=\kappa[\phi(w)].
\]
The thesis follows deriving in $\theta$.

\medskip\noindent
{\it $\bullet$ Step 3:  ${F}\in C^1$.}\\
We now focus on the proof of the $C^1$ regularity.  We now compute the partial derivative w.r.t. $f$ of ${F}$, i.e.

We now compute
\[
\pa_f F(\beta,V,{\mu},f)=\frac{d}{d\eps}F_\eps(\beta,V,{\mu},f)(\theta)\biggr|_{\eps=0}
\]
being   $F_\eps$ defined as
\begin{align*}
F_\eps(\beta,V,{\mu},f)(\theta)=\mbox{Re}&\left[Vw\pare{\phi'(w)+\eps h'(w)}+\frac{1}{4}\beta\partial_\theta \mathcal{H}[\kappa[\phi+\eps h]](\theta)-\frac14\partial_\theta\mathcal{H}[|\phi(w)+\eps h(w)|^2](\theta) \right.\\
& \left.  +\frac12 w \overline{\pare{\phi(w)+\eps h(w)}}\pare{\phi'(w)+\eps h'(w)}-\frac12 |\phi'(w)+\eps h'(w)|\right].
\end{align*}
We have that
\begin{align*}
\frac{d}{d\eps}Vw\pare{\phi'(w)+\eps h'(w)}&=Vwh'(w),\\
\frac{d}{d\eps}|\phi(w)+\eps h(w)|^2&=\frac{d}{d\eps}\pare{|\phi(w)|^2+2\eps\t{Re}\bra{\overline{\phi(w)}h(w)}+\eps^2| h(w)|^2}\\
\frac{d}{d\eps} w \overline{\pare{\phi(w)+\eps h(w)}}\pare{\phi'(w)+\eps h'(w)}&=w\pare{\overline{h(w)}\pare{\phi'(w)+\eps h'(w)}+h'(w)\overline{\pare{\phi(w)+\eps h(w)}}}\\
\frac{d}{d\eps}|\phi'(w)+\eps h'(w)|&=\frac{\phi'(w)+\eps h'(w)}{|\phi'(w)+\eps h'(w)|}\cdot h'(w),
\end{align*}
from which
\begin{align*}
%\frac{d}{d\eps}Vw\pare{\phi'(w)+\eps h'(w)}&\biggr|_{\eps=0}=Vwh'(w),\\
%
%
\frac{d}{d\eps}|\phi(w)+\eps h(w)|^2&\biggr|_{\eps=0}=2 \t{Re}\bra{\overline{\phi(w)}h(w)},\\
\frac{d}{d\eps} w \overline{\pare{\phi(w)+\eps h(w)}}\pare{\phi'(w)+\eps h'(w)}
&\biggr|_{\eps=0}=w\pare{\overline{h(w)}\phi'(w)+h'(w)\overline{\phi(w)}},\\
\frac{d}{d\eps}|\phi'(w)+\eps h'(w)|
&\biggr|_{\eps=0}=\frac{\phi'(w)}{|\phi'(w)|}\cdot h'(w).
\end{align*}
We also need to derive
\begin{align*}
\frac{d}{d\eps}\kappa[\phi+\eps h](w)& =\frac{d}{d\eps}\frac{1}{\av{\phi'(w)+\eps h'(w)}}\t{Re}\bra{1+w\frac{\phi''(w)+\eps h''(w)}{\phi'(w)+\eps h'(w)}}\\
& =-\frac{\phi'(w)+\eps h'(w)}{\av{\phi'(w)+\eps h'(w)}^3}h'(w)\t{Re}\bra{1+w\frac{\phi''(w)+\eps h''(w)}{\phi'(w)+\eps h'(w)}}\\
&\q+\frac{1}{\av{\phi'(w)+\eps h'(w)}}\t{Re}\bra{w\frac{h''(w)(\phi'(w)+\eps h'(w))-h'(w)(\phi''(w)+\eps h''(w))}{\av{\phi'(w)+\eps h'(w)}^2}},
\end{align*}
which yields to
\begin{align*}
\frac{d}{d\eps}\kappa[\phi+\eps h](w)
\biggr|_{\eps=0}&=-\frac{\phi'(w)}{\av{\phi'(w)}^3}\cdot h'(w)\t{Re}\bra{1+w\frac{\phi''(w)}{\phi'(w)}}\\
&\q+\frac{1}{\av{\phi'(w)}}\t{Re}\bra{w\frac{h''(w)\phi'(w)-h'(w)\phi''(w)}{\av{\phi'(w)}^2}}\\
&=\kappa_0[\phi].
\end{align*}
The expression of $\partial_f {F}$ follows gathering the previous computations:
\begin{equation*}\label{eq:pafF}
\begin{aligned}
\partial_f  {F}(\beta,V,{\mu},f)[h]&=\t{Re}\left[
Vwh'(w)+\frac{1}{4}\beta\partial_\theta\mathcal{H}[\kappa_0[\phi]](\theta)-\frac{1}{2}\partial_\theta\mathcal{H}\bra{ \overline{\phi(w)}h(w)}(\theta)
\right. \\
&\left.\qq\q+\frac{1}{2}w\pare{\overline{h(w)}\phi'(w)+h'(w)\overline{\phi(w)}}-\frac{1}{2}\frac{\phi'(w)}{|\phi'(w)|}h'(w)
\right].
\end{aligned}
\end{equation*}

%We are ready to check if the terms appearing in $\partial_V {F}(\beta,V,f)\lambda$ and $\partial_f {F}(\beta,V,f)[h]$ have the desired regularity.\\
Since $f$, $h$ and $\phi$ belong to $X^\alpha$,  we have that all the terms composing $\partial_f  {F}(\beta,V,{\mu},f)[h]$, up to $\partial_\theta\mathcal{H}[\kappa_0[\phi]](\theta)$, are in $C^{2,\alpha}$. As far as $\partial_\theta\mathcal{H}[\kappa_0[\phi]](\theta)$ is concerned, we have that $\kappa_0[\phi]$ belongs to $C^{1,\alpha}$, hence $\partial_\theta\mathcal{H}[\kappa_0[\phi]](\theta)$ is in $C^{0,\alpha}$.

Using similar ideas, we deduce
\begin{align*}
\pa_\mu F(\beta,V,{\mu},f)
&=\mbox{Re}\left[Vw-\frac{1}{4}\beta \partial_\theta \mathcal{H}\bra{
\frac{1}{\av{\phi'(w)}^3}\mbox{Re}\bra{\phi'(w)+2w\phi''(w)}
}(\theta)-  \frac12 \partial_\theta\mathcal{H}[\phi(w)w](\theta) \right.\\
&\qq\left.  +1+{\mu}+\frac{wf(\overline{w})+f'(w)}{2}-\frac{\phi'(w)}{2|\phi'(w)|}\right].
\end{align*}
Since $f$ and $\phi$ belong to $X^\alpha$, with $\norm{f}_{X^\alpha}\ll1$, we have that all the terms composing $\partial_\mu  {F}(\beta,V,{\mu},f)[h]$, up to $\partial_\theta\mathcal{H}[\pa_\mu\kappa[\phi]](\theta)$, are in $C^{2,\alpha}$. As far as $\partial_\theta\mathcal{H}[\pa_\mu\kappa[\phi]](\theta)$ is concerned, we have that $\pa_\mu\kappa[\phi]$ belongs to $C^{1,\alpha}$, hence $\partial_\theta\mathcal{H}[\pa_\mu\kappa[\phi]](\theta)$ is in $C^{0,\alpha}$.

%Furthermore, Proposition  {prop:mu} provides us with the continuity of $\partial_f {\mu}(f)[h]$.
%Furthermore, since
%$$
%(\partial_\mu G)(\widehat{\mu},f)\partial_f \widehat{\mu}(f)[h]+\partial_f G(\widehat{\mu},f)[h]=0,
%$$
%with $G$ defined in \eqref{eq:G}, we deduce that
%\begin{align*}
%\pa_f\widehat{\mu}(f)[h]&=-\frac{\partial_f G(\widehat{\mu},f)[h]}{\partial_\mu G(\widehat{\mu},f)}
%\end{align*}
%with $\partial_f G(\widehat{\mu},f)[h]$ and $\partial_\mu G(\widehat{\mu},f)$ as in   \eqref{eq:pafG} and \eqref{eq:pamuG}, respectively.

Finally, let us compute the other derivatives. First, note that
\begin{equation}\label{eq:paVF}
\partial_V {F}(\beta,V,\mu, f)=\t{Re}\bra{w\pare{1+\mu+f'(w)}}.
\end{equation}
Since $f$ and $\phi$ belong to $X^\alpha$, we have that  $\partial_V {F}(\beta,V,\mu, f)$ belongs to in $C^{2,\alpha}$.
Second, we compute the derivative with respect to $\beta$:
$$
\partial_\beta F(\beta,V,\mu,f)(\theta)=\frac14 \textnormal{Re}\left[\partial_\theta \mathcal{H}[\kappa[\phi(w)]](\theta)\right],
$$
which belongs to $Y^\alpha$ following the ideas above.
\end{proof}

\subsection{Crandall-Rabinowitz theorem}
The goal of the work has been reduced to study the nontrivial roots of the nonlinear functional $F$. That is the main task of bifurcation theory. Here, we shall use the so-called Crandall-Rabinowiz theorem, which can be found in \cite{CR}.

\begin{theorem}[Crandall-Rabinowitz Theorem]\label{CR}
    Let $X, Y$ be two Banach spaces, $V$ be a neighborhood of $0$ in $X$ and $F:\mathbb{R}\times V\rightarrow Y$ be a function with the properties,
    \begin{enumerate}  \setlength\itemsep{.5em}
        \item[(CR1)] $F(\lambda,0)=0$ for all $\lambda\in\mathbb{R}$.
        \item[(CR2)] The partial derivatives  $\partial_\lambda F_{\lambda}$, $\partial_fF$ and  $\partial_{\lambda}\partial_fF$ exist and are continuous.
        \item[(CR3)] The operator $\partial_f F(\lambda_0,0)$ is Fredholm of zero index and $\textnormal{Ker}(F_f(\lambda_0,0))=\langle f_0\rangle$ is one-dimensi\-onal.
                \item[(CR4)]  Transversality assumption: $\partial_{\lambda}\partial_fF(\lambda_0,0)f_0 \notin \textnormal{Im}(\partial_fF(\lambda_0,0))$.
    \end{enumerate}
    If $Z$ is any complement of  $\textnormal{Ker}(\partial_fF(\lambda_0,0))$ in $X$, then there is a neighborhood  $U$ of $(\lambda_0,0)$ in $\mathbb{R}\times X$, an interval  $(-a,a)$, and two continuous functions $\Phi:(-a,a)\rightarrow\mathbb{R}$, $\beta:(-a,a)\rightarrow Z$ such that $\Phi(0)=\lambda_0$ and $\beta(0)=0$ and
    $$F^{-1}(0)\cap U=\{(\Phi(s), s f_0+s\beta(s)) : |s|<a\}\cup\{(\lambda,0): (\lambda,0)\in U\}.$$
\end{theorem}

We recall that $F$ a Fredholm operator of zero index if $\textnormal{Ker}(F_f(\lambda_0,0))$ is a one-dimensional subspace of $\mathbb{R}\times V$, and $\textnormal{Range}(F_f(\lambda_0,0))$ is a closed subspace of $Y$ of codimension one. In this context, we will say that $\lambda_0$ is an eigenvalue of $F$.

\section{Spectral study}

In order to apply the Crandall-Rabinowitz Theorem  {CR}, we should check the third and fourth hypothesis which are related to the spectral study of ${F}$. Since $\partial_f F(\beta,0,0,0)$ does not satisfied the hypothesis of Crandall-Rabinowitz theorem, we should also work with the free parameters $(V,\mu)$. Let us compute its linearized operator around the trivial solution.

\begin{proposition}\label{prop:paF}
The linearized operator of ${F}:\R^3\times   {B_{X^\alpha}(\rho)}\rightarrow Y^\alpha$  reads as
\begin{align*}
&D_{(\mu, V,f)}{F}(\beta,0,0,0)[\lambda_1,\lambda_2,h](\theta)\\
&=\textnormal{Re}\left[\frac12 \lambda_1+\lambda_2 w+\frac14 \beta \partial_\theta \mathcal{H}[\textnormal{Re}(wh''(w))](\theta)-\frac14 \beta \partial_\theta \mathcal{H}[\mbox{Re}(h'(w))](\theta)\right.\\
&\q\left.-\frac12 \partial_\theta \mathcal{H}[\textnormal{Re}(\overline{w}h(w))](\theta)+\frac12 w\overline{h(w)})\right].
\end{align*}
\end{proposition}

\begin{proof}
  {
The linearized operator of $ {F}$ at $(\beta,0,0,0)$ is given by
\[
D_{(\mu, V,f)} {F}(\beta,0,0,0)[\lambda_1,\lambda_2,h]=\partial_f F(\beta,0,0,0)[h]+\partial_\mu F(\beta,0,0,0)\lambda_1+\partial_V F(\beta,0,0,0)\lambda_2,
\]
where $\partial_f F$ can be found in \eqref{eq:paVF}. For the other two derivatives, note that
$$
F(\beta,V,0,0)(\theta)=V\mbox{Re}[w],\quad F(\beta,0,\mu,0)(\theta)=\frac12 (1+\mu)\mu,\quad w=e^{i\theta},
$$
which implies
$$
\partial_V F(\beta,0,0,0)[\lambda_2](\theta)=\lambda_2 \mbox{Re}[w],\quad \partial_\mu F(\beta,0,0,0)[\lambda_1](\theta)=\frac12 \lambda_1.
$$

}
%\begin{align*}
%\partial_f F(\beta,0,0)[h](\theta)=&\mbox{Re}\left[\frac14 \beta \partial_\theta \mathcal{H}[\partial_\phi\kappa[\mbox{Id}][h]](\theta)-\frac12 \partial_\theta \mathcal{H}[\mbox{Re}(  {\overline{w}h(w)})](\theta) +\frac12 w\overline{h(w)}\right]
%\end{align*}
%Notice that
%$$
%\partial_\phi \kappa[\mbox{Id}][h](w)=-\mbox{Re}(h'(w))+\mbox{Re}(wh''(w)),
%$$
%and then
%\begin{align*}
%\partial_f F(\beta,0,0)[h](\theta)=\mbox{Re}&\left[\frac14 \beta \partial_\theta \mathcal{H}[\mbox{Re}(wh''(w))](\theta)-\frac14 \beta \partial_\theta \mathcal{H}[\mbox{Re}(h'(w))](\theta)-\frac12 \partial_\theta \mathcal{H}[\mbox{Re}(\overline{w}h(w))](\theta)\right.\\
%&\left.+\frac12 w\overline{h(w)})\right].
%\end{align*}

\end{proof}

In the following proposition, we write the linearized operator in Fourier series
\begin{proposition}\label{prop-fourier}
The linearized operator of $ {F}:\R^3 \times   {B_{X^\alpha}(\rho)}\rightarrow Y^\alpha$  in Fourier series agrees with
\begin{align*}
\mathcal{L}[\beta](\lambda_1,\lambda_2,h):=&\partial_\mu F(\beta,0,0,0)\lambda_1+\partial_V  {F}(\beta,0,0,0)\lambda_2+\partial_f  {F}(\beta,0,0,0)[h]\\
&=\frac12 \lambda_1+\lambda_2\cos(\theta)+\frac{1}{4}\sum_{n\geq 2}a_n n(n+1)(n-1)\cos(n\theta)\left\{\beta-\beta_n\right\},
\end{align*}
where
$$
\beta_n:=\frac{2}{n(n+1)},\quad n\geq 2.
$$
\end{proposition}

\begin{proof}
Take $h(w)=\sum_{n\geq 2}a_n w^{n+1}$, then $h'(w)=\sum_{n\geq 2}a_n (n+1) w^n$. Let us start with the terms involving the Hilbert transform.

Using the ideas of \cite[Section 9.3.1, eq. (9.3.8)]{BN}, we can write the following:
\begin{align*}
\partial_\theta \mathcal{H}[\mbox{Re}(h'(w))](\theta)=&\frac{1}{2\pi}\partial_\theta \int_0^{2\pi}\mbox{Re}(h'(e^{is}))\cot\pare{\theta-s}ds\\
=&\partial_\theta \frac{1}{2\pi}\sum_{n\geq 1} a_n (n+1)\int_0^{2\pi}\sin(ns)\cot\pare{\theta-s}ds\\
=&\partial_\theta \sum_{n\geq 1}a_n(n+1)\sin(n\theta)\\
=&\sum_{n\geq 2} a_n n(n+1)\cos(n\theta).
\end{align*}
On the other hand
\begin{align*}
\partial_\theta \mathcal{H}[\mbox{Re}(wh''(w)))(\theta)=&\partial_\theta \sum_{n\geq 1} a_n n (n+1)\sin(n\theta)\\
=&\sum_{n\geq 2} a_n n^2 (n+1)\cos(n\theta)
\end{align*}
Then
\begin{align*}
\partial_f  {F}(\beta,0,0)[h](\theta)=&\sum_{n\geq 2}a_n\cos(n\theta)\left\{\frac14 \beta n^2(n+1)-\frac14 \beta n(n+1)-\frac12 n+\frac12\right\}\\
=&\frac14\sum_{n\geq 2}a_n n(n+1)(n-1)\cos(n\theta)\left\{\beta-\beta_n\right\}.
\end{align*}

On the other hand
$$
\partial_V  {F}(\beta,0,0)\lambda_2=\lambda_2\cos(\theta),
$$
and thus we achieve the result.
\end{proof}

\subsection{Kernel and range study} \label{sec:kernel}
We verify that the assumption (CR3) of Theorem  {CR} is satisfied. More precisely, we prove that,
choosing $\beta=\beta_m$, we get a one dimensional kernel generated by $(0,0,w^{m+1})$. On the other hand, $Y\backslash \mbox{Range}$ is generated by $\cos(m\theta)$.

\begin{proposition}\label{prop:kernel}
The kernel and the range of the linearized operator agrees with
$$
\textnormal{Ker}\, \mathcal{L}[\beta_m]=<(0,0,w^{m+1})>,
$$
and
$$
\textnormal{Range}\, \mathcal{L}[\beta_m]=\left\{f\in C^{0,\alpha}(\T), \quad f(e^{i\theta})=\sum_{m\neq n\geq 0}a_n \cos(n\theta), \quad  a_n\in\R\right\}.
$$
\end{proposition}
\begin{proof}
The description of the kernel comes from the expression of the linearized operator in Fourier in Proposition  \ref{prop-fourier}.

Let us study the range. Note that
$$
\textnormal{Range}\, \mathcal{L}[\beta_m] \subset  \left\{f\in C^{0,\alpha}(\T), \quad f(e^{i\theta})=\sum_{m\neq n\geq 0}a_n \cos(n\theta), \quad  a_n\in\R\right\}.
$$
Let us check the other inclusion. Take
$$
g_0\in \left\{f\in C^{0,\alpha}(\T), \quad f(e^{i\theta})=\sum_{m\neq n\geq 0}a_n \cos(n\theta), \quad  a_n\in\R\right\},
$$
and let us check that it has a preimage. We can write $g$ as
$$
g_0(e^{i\theta})=\sum_{m\neq n\geq 0}d_n \cos(n\theta).
$$
Then, the equation
$$
\mathcal{L}[\beta_m](\lambda_1,\lambda_2,h)=g_0,
$$
implies
$$
\lambda_1=2d_0,\quad \lambda_2=d_1, \quad a_n=\frac{4}{n(n+1)(n-1)(\beta_m-\beta_n)}d_n, \, n\geq 2, n\neq m.
$$
Then, the candidate to preimage is
$$
\lambda_1=2d_0,\quad \lambda_2=d_1, \quad h_0(w)=\sum_{m\neq n\geq 2}\frac{4}{n(n+1)(n-1)(\beta_m-\beta_n)}d_n w^{n+1}.
$$
It remains to check that $h_0\in X$, and in particular, $h_0\in C^{3,\alpha}(\T)$. Notice that
$$
h_0'''(w)=\sum_{m\neq n\geq 2}\frac{4}{(\beta_m-\beta_n)}d_n w^{n-2}.
$$
Recalling that $|w|=1$, and summing and subtracting $1/\beta_m$, it can be written as
\begin{align*}
h_0'''(w)=&4\overline{w}^3\sum_{m\neq n\geq 2}\left(\frac{1}{(\beta_m-\beta_n)}-\frac{1}{\beta_m} \right) d_n w^{n+1}+4\frac{1}{\beta_m}\overline{w}^3\sum_{m\neq n\geq 2}d_n w^{n+1}\\
=&4\overline{w}^3\sum_{m\neq n\geq 2}\frac{\beta_n}{(\beta_m-\beta_n)\beta_m} d_n w^{n+1}+4\frac{1}{\beta_m}\overline{w}^3\sum_{m\neq n\geq 2}d_n w^{n+1}.
\end{align*}
The second term is clearly in $C^{0,\alpha}$ since $g_0\in C^{0,\alpha}$. The first term can be written as a convolution:
$$
\sum_{m\neq n\geq 2}\frac{\beta_n}{(\beta_m-\beta_n)\beta_m} d_n w^{n+1}=K\star g_0,
$$
where
$$
K(w)=\sum_{m\neq n\geq 2}\frac{\beta_n}{(\beta_m-\beta_n)\beta_m} w^{n+1}.
$$
Since we have a convolution, and $g_0\in C^{0,\alpha}$, we have that the first term belongs to $C^{0,\alpha}$ if $K\in L^1$. Indeed, we can check that $K\in L^2$ using Parseval's inequality:
$$
\norm{K}_2^2=\sum_{m\neq n\geq 2}\frac{\beta_n^2}{(\beta_m-\beta_n)^2\beta_m^2}\leq C\sum _{m\neq n\geq 2}\frac{1}{n^4}\leq C,
$$
by using the decay of $\beta_n$. That concludes the proof.
\end{proof}

\subsection{Transversal condition} \label{sec:transversal}
Here, we aim to prove the transversal condition (CR4) of the Crandall-Rabinowitz Theorem  {CR}.

\begin{proposition}
For $\beta=\beta_m$, with $m\geq 2$, the transversal condition is satisfied, that is
$$
\partial_\beta \partial_{(\mu,V,f)}  {F}(\beta_m,0,0,0)[0,0,w^{m+1}]\notin \textnormal{Range}\, \mathcal{L}[\beta_m].
$$

\end{proposition}
\begin{proof}
Notice that
\begin{align*}
\partial_\beta \partial_{(\mu,V,f)}  {F}(\beta_m,0,0,0)[0,0,w^{m+1}]= \frac14 m(m+1)(m-1)\cos(m\theta),
\end{align*}
which does not belong to the range.
\end{proof}

\section{Main result}
Finally, we state a detailed version of Theorem \ref{theo:main-intro}, which is the main goal of our work. For that, let us modify the spaces \eqref{eq:X}-\eqref{eq:Y} by adding the $m$-fold symmetry:
\begin{align*}
X^\alpha_m:=&\left\{f\in C^{3,\alpha}(\T), \quad f(w)=\sum_{n\geq 2}a_{nm} w^{nm+1}, \quad a_{nm}\in\R\right\}\\%\label{eq:Xm}\\
Y^\alpha_m:=&\left\{f\in C^{0,\alpha}(\T), \quad f(e^{i\theta})=\sum_{n\geq 0}a_{nm} \cos(nm\theta), \quad a_{nm}\in\R\right\},%\label{eq:Ym}
\end{align*}
where $m\in\N$. Note that we can use all the work done in the previous sections just changing $n$ by $nm$.

\begin{theorem}\label{theo:main}
Let $m\geq 2$,
\[
\beta_m=\frac{2}{m(m+1)},
\]
and
$$
\phi(w)=(1+\mu)w+f(w), \quad w\in\T,
$$
 which maps $\T$ into some boundary $\partial D$. Then, there exists $\varepsilon>0$ and continuous curves $\xi\in (-\varepsilon,\varepsilon) \mapsto (\beta(\xi),V(\xi),\mu(\xi),f(\xi))$ such that
\[
F\pare{\beta(\xi),V(\xi),\mu(\xi),f(\xi)}(\theta)=0\qquad\forall\theta\in[0,2\pi].
\]
Moreover $\beta(0)=\beta_m$, $V(0)=\mu(0)=0$ and $f(0)=w^{m+1}$. Hence, the associated domain $D_{\xi}$ describes a $m$-fold symmetric traveling wave with constant speed $V(\xi)$.
\end{theorem}
\begin{proof}
The proof is based on the application of the Crandall-Rabinowitz theorem to $F$.

By Proposition \ref{prop:welldefined} we know that
$$
F:\R^3\times B_{X^\alpha_m}(\rho)\mapsto Y^\alpha_m,
$$
is well-defined and $C^1$ for $m\geq 1$, $\alpha\in(0,1)$, and $\rho<1$. Moreover, Proposition \ref{prop:0} implies that
$$
F(\beta,0,0,0)\equiv 0, \quad \forall\beta.
$$
Using Proposition \ref{prop:kernel} and Proposition \ref{prop:trans} we know that the linearized operator is Fredholm with one dimensional kernel, and the transversal condition is satisfied.
\end{proof}

\section*{Acknowledgments}
C.G. has been supported by RYC2022-035967-I (MCIU/AEI/10.13039/501100011033 and FSE+), and partially by Grants PID2022-140494NA-I00 and PID2022-137228OB-I00 funded by MCIN/AEI/10.13039/501100011033/FEDER, UE, by Grant C-EXP-265-UGR23 funded by Consejeria de Universidad, Investigacion e Innovacion \& ERDF/EU Andalusia Program, and by Modeling Nature Research Unit, project QUAL21-011.\\
The work of M.M. was partially supported by Grant RYC2021-033698-I, funded by the Ministry of Science and Innovation/State Research Agency/10.13039/501100011033, by the European Union ”NextGenerationEU/Recovery, Transformation and Resilience Plan”; by the project ”Análisis Matemático Aplicado y Ecuaciones Diferenciales” Grant PID2022-141187NB- I00 funded by MCIN /AEI /10.13039/501100011033 / FEDER, UE and acronym ”AMAED”; by the project PID2022-140494NA-I00 funded by MCIN/AEI/10.13039/501100011033/FEDER, UE.

\end{document}